\documentclass[reqno, 12pt]{amsart}
\pdfoutput=1
\makeatletter
\let\origsection=\section \def\section{\@ifstar{\origsection*}{\mysection}} 
\def\mysection{\@startsection{section}{1}\z@{.7\linespacing\@plus\linespacing}{.5\linespacing}{\normalfont\scshape\centering\S}}
\makeatother        

\usepackage{amsmath,amssymb,amsthm}
\usepackage{mathrsfs}
\usepackage{mathabx}\changenotsign
\usepackage{dsfont}

\usepackage{xcolor}
\usepackage[backref=page]{hyperref}
\hypersetup{
    colorlinks,
    linkcolor={red!60!black},
    citecolor={green!60!black},
    urlcolor={blue!60!black}
}

\usepackage[open,openlevel=2,atend]{bookmark}

\usepackage[abbrev,msc-links,backrefs]{amsrefs} 
\usepackage{doi}

\renewcommand{\PrintDOI}[1]{\doi{#1}}

\usepackage[T1]{fontenc}
\usepackage{lmodern}
\usepackage[babel]{microtype}

\usepackage[english]{babel}

\usepackage{geometry}
\usepackage{enumitem}
\def\rmlabel{\upshape({\itshape \roman*\,})}

\def\alabel{\upshape({\itshape \alph*\,})}

\def\nlabel{\upshape({\itshape \arabic*\,})}

\let\polishlcross=\l
\def\l{\ifmmode\ell\else\polishlcross\fi}

\def\qand{\quad\text{and}\quad}
\def\qqand{\qquad\text{and}\qquad}

\let\emptyset=\varnothing
\let\setminus=\smallsetminus

\makeatletter
\def\moverlay{\mathpalette\mov@rlay}
\def\mov@rlay#1#2{\leavevmode\vtop{%
   \baselineskip\z@skip \lineskiplimit-\maxdimen
   \ialign{\hfil$\m@th#1##$\hfil\cr#2\crcr}}}
\newcommand{\charfusion}[3][\mathord]{
    #1{\ifx#1\mathop\vphantom{#2}\fi
        \mathpalette\mov@rlay{#2\cr#3}
      }
    \ifx#1\mathop\expandafter\displaylimits\fi}
\makeatother

\newcommand{\dcup}{\charfusion[\mathbin]{\cup}{\cdot}}

\usepackage{tikz}

\newlength{\myscaledsize}

\newcommand{\vdeg}{	\setlength{\myscaledsize}{\the\fontdimen6\font}
	\hspace*{-0.08em}\resizebox{0.5\myscaledsize}{0.4\myscaledsize}{
		\tikz{
									\draw[black] (0,0) circle (.2);
			\draw[black] (1,0) circle (.2);
			\draw[black,fill=black] (0.5,0.86) circle (.2); 
		}}
}
\newcommand{\VDEG}{	\setlength{\myscaledsize}{\the\fontdimen6\font}
	\hspace*{-0.0444em}\raisebox{-0.1\myscaledsize}{\resizebox{0.75\myscaledsize}{0.6\myscaledsize}{
		\tikz{
									\draw[black] (0,0) circle (.2);
			\draw[black] (1,0) circle (.2);
			\draw[black,fill=black] (0.5,0.86) circle (.2); 
		}}}
		\hspace{0.1em}
}
\newcommand{\pivdeg}{\pi_{\vdeg}}

\newcommand{\vv}{	\setlength{\myscaledsize}{\the\fontdimen6\font}
	\hspace*{-0.08em}\resizebox{0.5\myscaledsize}{0.4\myscaledsize}{
		\tikz{
			\draw[black,fill=black] (0,0) circle (.2);
			\draw[black,fill=black] (1,0) circle (.2);
						\draw[black] (0.5,0.86) circle (.2);
		}}
}
\newcommand{\VV}{	\setlength{\myscaledsize}{\the\fontdimen6\font}
	\hspace*{-0.0444em}\raisebox{-0.1\myscaledsize}{\resizebox{0.75\myscaledsize}{0.6\myscaledsize}{
		\tikz{
			\draw[black,fill=black] (0,0) circle (.2);
			\draw[black,fill=black] (1,0) circle (.2);
						\draw[black] (0.5,0.86) circle (.2);
		}}}
		\hspace{0.1em}
}
\newcommand{\pivv}{\pi_{\vv}}

\newcommand{\vvv}{	\setlength{\myscaledsize}{\the\fontdimen6\font}
	\hspace*{-0.08em}\resizebox{0.5\myscaledsize}{0.4\myscaledsize}{
		\tikz{
			\draw[black,fill=black] (0,0) circle (.2);
			\draw[black,fill=black] (1,0) circle (.2);
			\draw[black,fill=black] (0.5,0.86) circle (.2);}}
}
\newcommand{\VVV}{	\setlength{\myscaledsize}{\the\fontdimen6\font}
	\hspace*{-0.0444em}\raisebox{-0.1\myscaledsize}{\resizebox{0.75\myscaledsize}{0.6\myscaledsize}{
		\tikz{
			\draw[black,fill=black] (0,0) circle (.2);
			\draw[black,fill=black] (1,0) circle (.2);
			\draw[black,fill=black] (0.5,0.86) circle (.2);}}}
			\hspace{0.1em}
}
\newcommand{\pivvv}{\pi_{\vvv}}

\newcommand{\epair}{	\setlength{\myscaledsize}{\the\fontdimen6\font}
	\hspace*{-0.08em}\resizebox{0.5\myscaledsize}{0.4\myscaledsize}{
		\tikz{
			\draw[black,fill=black] (0,0) circle (.2);
			\draw[black,fill=black] (1,0) circle (.2);
						\draw[black] (0.5,0.86) circle (.2);
			\draw[black,line width=4pt ](0,0)--(1,0);}}
}
\newcommand{\EPAIR}{	\setlength{\myscaledsize}{\the\fontdimen6\font}
	\hspace*{-0.0444em}\raisebox{-0.1\myscaledsize}{\resizebox{0.75\myscaledsize}{0.6\myscaledsize}{
		\tikz{
			\draw[black,fill=black] (0,0) circle (.2);
			\draw[black,fill=black] (1,0) circle (.2);
						\draw[black] (0.5,0.86) circle (.2);
			\draw[black,line width=4pt ](0,0)--(1,0);}}}
			\hspace{0.1em}
}
\newcommand{\piepair}{\pi_{\epair}}

\newcommand{\ev}{	\setlength{\myscaledsize}{\the\fontdimen6\font}
	\hspace*{-0.08em}\resizebox{0.5\myscaledsize}{0.4\myscaledsize}{
		\tikz{
			\draw[black,fill=black] (0,0) circle (.2);
			\draw[black,fill=black] (1,0) circle (.2);
			\draw[black,fill=black] (0.5,0.86) circle (.2);
			\draw[black,line width=4pt ](0,0)--(1,0);}}
}
\newcommand{\EV}{	\setlength{\myscaledsize}{\the\fontdimen6\font}
	\hspace*{-0.0444em}\raisebox{-0.1\myscaledsize}{\resizebox{0.75\myscaledsize}{0.6\myscaledsize}{
		\tikz{
			\draw[black,fill=black] (0,0) circle (.2);
			\draw[black,fill=black] (1,0) circle (.2);
			\draw[black,fill=black] (0.5,0.86) circle (.2);
			\draw[black,line width=4pt ](0,0)--(1,0);}}}
			\hspace{0.1em}
}
\newcommand{\piev}{\pi_{\ev}}
\let\pilp=\piev

\let\LP=\EV

\newcommand{\ee}{	\setlength{\myscaledsize}{\the\fontdimen6\font}
	\hspace*{-0.08em}\resizebox{0.5\myscaledsize}{0.4\myscaledsize}{
		\tikz{
			\draw[black,fill=black] (0,0) circle (.2);
			\draw[black,fill=black] (1,0) circle (.2);
			\draw[black,fill=black] (0.5,0.86) circle (.2);
			\draw[black,line width=4pt ](0,0)--(0.5,0.86);
			\draw[black,line width=4pt ](1,0)--(0.5,0.86);}}
}
\newcommand{\EE}{	\setlength{\myscaledsize}{\the\fontdimen6\font}
	\hspace*{-0.0444em}\raisebox{-0.1\myscaledsize}{\resizebox{0.75\myscaledsize}{0.6\myscaledsize}{
		\tikz{
			\draw[black,fill=black] (0,0) circle (.2);
			\draw[black,fill=black] (1,0) circle (.2);
			\draw[black,fill=black] (0.5,0.86) circle (.2);
			\draw[black,line width=4pt ](0,0)--(0.5,0.86);
			\draw[black,line width=4pt ](1,0)--(0.5,0.86);}}}
			\hspace{0.1em}
}
\newcommand{\piee}{\pi_{\ee}}
\let\pill=\piee
\let\QLL=\EE

\newcommand{\eee}{	\setlength{\myscaledsize}{\the\fontdimen6\font}
	\hspace*{-0.08em}\resizebox{0.5\myscaledsize}{0.4\myscaledsize}{
		\tikz{
			\draw[black,fill=black] (0,0) circle (.2);
			\draw[black,fill=black] (1,0) circle (.2);
			\draw[black,fill=black] (0.5,0.86) circle (.2);
			\draw[black,line width=4pt ](0,0)--(1,0);
			\draw[black,line width=4pt ](0,0)--(0.5,0.86);
			\draw[black,line width=4pt ](1,0)--(0.5,0.86);}}
}
\newcommand{\EEE}{	\setlength{\myscaledsize}{\the\fontdimen6\font}
	\hspace*{-0.0444em}\raisebox{-0.1\myscaledsize}{\resizebox{0.75\myscaledsize}{0.6\myscaledsize}{
		\tikz{
			\draw[black,fill=black] (0,0) circle (.2);
			\draw[black,fill=black] (1,0) circle (.2);
			\draw[black,fill=black] (0.5,0.86) circle (.2);
			\draw[black,line width=4pt ](0,0)--(1,0);
			\draw[black,line width=4pt ](0,0)--(0.5,0.86);
			\draw[black,line width=4pt ](1,0)--(0.5,0.86);}}}
			\hspace{0.1em}
}

\let\pilll=\pieee
\let\LLL=\EEE

\theoremstyle{plain}
\newtheorem{thm}{Theorem}[section]

\newtheorem{lemma}[thm]{Lemma}

\theoremstyle{definition}

\newtheorem{dfn}[thm]{Definition}
\newtheorem{exmp}[thm]{Example}
\newtheorem{quest}{Question}[section]

\let\eps=\varepsilon
\let\epsilon=\varepsilon
\let\theta=\vartheta
\let\rho=\varrho
\let\phi=\varphi

\def\NN{\mathds N}

\def\RR{\mathds R}
 
\def\cA{{\mathcal A}}

\def\cP{{\mathcal P}}
\def\cQ{{\mathcal Q}}

\def\cK{{\mathcal K}}

\def\ccA{{\mathscr{A}}}

\DeclareMathOperator{\ex}{ex}
\DeclareMathOperator{\codeg}{codeg}

\def\bl{\bigl(}
\def\br{\bigr)}

\let\setminus=\smallsetminus
\let\emptyset=\varnothing

\begin{document}
\title[Embedding tetrahedra into quasirandom hypergraphs]{Embedding tetrahedra into quasirandom hypergraphs}

\author[Christian Reiher]{Christian Reiher}
\address{Fachbereich Mathematik, Universit\"at Hamburg, Hamburg, Germany}
\email{Christian.Reiher@uni-hamburg.de}

\author[Vojt\v{e}ch R\"{o}dl]{Vojt\v{e}ch R\"{o}dl}
\address{Department of Mathematics and Computer Science, 
Emory University, Atlanta, USA}
\email{rodl@mathcs.emory.edu}
\thanks{The second author was supported by NSF grants DMS 1301698 and 1102086.}

\author[Mathias Schacht]{Mathias Schacht}
\address{Fachbereich Mathematik, Universit\"at Hamburg, Hamburg, Germany}
\email{schacht@math.uni-hamburg.de}
\thanks{The third author was supported through the Heisenberg-Programme of the DFG}

\begin{abstract}
We investigate extremal problems for quasirandom hypergraphs. We say 
that a $3$-uniform hypergraph $H=(V,E)$ is
\emph{$(d,\eta,\LP)$-quasirandom}
 if for any subset 
$X\subseteq V$ and every set of pairs $P\subseteq V\times V$
the number of pairs $(x,(y,z))\in X\times P$ with
$\{x,y,z\}$ being a hyperedge of $H$ is in the interval $d\,|X|\,|P|\pm\eta\,|V|^3$. 
We show that for any $\eps>0$ there exists  $\eta>0$ such that 
every sufficiently large $(1/2+\eps,\eta,\LP)$-quasirandom hypergraph 
contains a tetrahedron, i.e., four vertices spanning all four hyperedges.
A known random construction shows that the density $1/2$ is best possible.
This result is closely related to a question of Erd\H os, whether every weakly quasirandom 
$3$-uniform hypergraph~$H$ with density bigger than $1/2$, 
i.e., every large subset of vertices induces a hypergraph with density bigger than~$1/2$, contains a tetrahedron. 
\end{abstract} 

\maketitle

\section{Introduction}  \label{sec:intro}
\subsection{Extremal problems for graphs and hypergraphs} Given a fixed graph $F$ a typical problem in extremal graph theory asks for the maximum number of edges that a (large) graph~$G$ on $n$ vertices containing no copy of $F$ can have. More formally, for a fixed graph~$F$ let the \emph{extremal number $\ex(n, F)$} be the number $|E|$ 
of edges of an  $F$-free graph~$G=(V,E)$ on $|V|=n$ vertices with the maximum number of edges. It is well known and not hard to observe that the sequence 
$
\ex(n, F)/\binom{n}{2}
$
is decreasing. Consequently, one may define the \emph{Tur\'an density}
\[
\pi(F)=\lim_{n\to\infty}  \frac{\ex(n, F)}{\binom{n}{2}},
\]
which describes the maximum density of large $F$-free graphs. The systematic study of these extremal parameters was initiated 
by Tur\'an~\cite{Tu41}, who determined $\ex(n,K_k)$ for complete graphs $K_k$. 
Thanks to his work and the results from \cite{ErSt46} by Erd\H{o}s and Stone it is known that 
the Tur\'an density of any graph $F$ with at least one edge can be explicitly computed using the 
formula
\begin{equation}\label{ES-formula}
\pi(F)=\frac{\chi(F)-2}{\chi(F)-1}\,.
\end{equation}

Already in his original work~\cite{Tu41} Tur\'an asked for hypergraph extensions of these extremal problems.
We restrict ourselves to \emph{$3$-uniform hypergraphs $H=(V,E)$}, where $V=V(H)$ is a finite set of \emph{vertices}
and the set of \emph{hyperedges} $E=E(H)\subseteq \binom{V}{3}$ is a collection of $3$-element sets of vertices. We shall only consider graphs and $3$-uniform hypergraphs and when we are referring  simply to a 
hypergraph we will always mean a $3$-uniform hypergraph. Moreover, for a simpler notation in the context of edges $\{i,j\}$ and 
hyperedges $\{i,j,k\}$ we omit the parentheses and just write $ij$ or $ijk$. In particular, $ijk$ denotes
an unordered triple, while for ordered triples we stick to the standard notation $(i,j,k)$.

Despite considerable effort
no formula similar to \eqref{ES-formula} is known or conjectured to hold for 
general $3$-uniform hypergraphs $F$. 
Determining the value of $\pi(F)$ is a well known and hard problem even 
for ``simple'' hypergraphs like the complete $3$-uniform hypergraph~$K_4^{(3)}$ on four vertices, which 
is also called the \emph{tetrahedron}. Currently the best known bounds for its Tur\'an density are 
\[
	\frac{5}{9}\leq\pi(K_4^{(3)})\leq 0.5616\,,
\] 
where the lower bounds is given by what is believed to be an optimal construction due to Tur\'an 
(see, e.g.,~\cite{Er77}). The upper bound is due to Razborov~\cite{Ra10} 
and the proof is based on the \emph{flag algebra method} introduced by Razborov~\cite{Ra07}. For a thorough discussion of Tur\'an type 
results and problems for hypergraphs we refer to the recent survey of Keevash~\cite{Ke11}.

\subsection{Quasirandom graphs and hypergraphs}
We consider a variant of Tur\'an type questions in connection with quasirandom hypergraphs. Roughly speaking,
a quasirandom hypergraph ``resembles'' a random hypergraph of the same edge density, by 
sharing some of the key properties with it, i.e., properties that hold for the random  
hypergraph with probability close to $1$.

The 
investigation of quasirandom graphs was initiated with the observation
that several such properties of randomly generated graphs 
are equivalent in a deterministic sense. This phenomenon turned 
out to be useful and had a number of applications in combinatorics. 
The systematic study of quasirandom graphs was initiated  by Thomason~\cites{Th87a,Th87b} and by Chung, Graham, and Wilson~\cite{CGW89}. 
A pivotal feature of random graphs is the uniform edge distribution on ``large'' sets of vertices and a quantitative version
of this property is used to define quasirandom graphs. More precisely, a graph $G=(V, E)$ is \emph{quasirandom with density 
$d>0$} if for every subset of vertices $U\subseteq V$ the number $e(U)$ of edges contained in~$U$ satisfies 
\begin{equation}\label{eq:qrG}
e(U)=d\tbinom{|U|}{2}+o(|V|^2)\,,
\end{equation}
where $o(|V|^2)/|V|^2\to0$ as $|V(G)|$ tends to infinity. Strictly speaking, we consider here 
a sequence of graphs $G_n=(V_n,E_n)$ where the number of vertices $|V_n|$ tends to infinity, but for the sake of 
a simpler presentation we will suppress the sequence in the discussion here. 
The main result in~\cite{CGW89} asserts, that satisfying~\eqref{eq:qrG} 
is deterministically equivalent to several other important properties of random graphs. In particular, it implies 
that for any fixed graph $F$ with $v_F$ vertices and $e_F$ edges
the number $N_F(G)$ of labeled copies of $F$ in a quasirandom graph $G=(V,E)$ of density $d$
satisfies
\begin{equation}\label{eq:qrN}
	N_F(G)=d^{e_F}|V|^{v_F}+o(|V|^{v_F})\,.
\end{equation}
In other words, the number of copies of $F$ is close to the expected value in a random graph with edge density $d$.

The analogous statement for hypergraphs fails to be true and uniform edge distribution on vertex sets
is not sufficient to enforce a property similar
to~\eqref{eq:qrN} for all fixed $3$-uniform hypergraphs $F$ (see, e.g., Example~\ref{VR-graph} below). 
A stronger notion of quasirandomness for which such an embedding result actually is true,
was considered in connection with the regularity method for hypergraphs (cf.\ Section~\ref{sec:regmethod}).

In fact below we consider four different notions of quasirandom hypergraphs. The first and weakest concept that we consider here is $(d,\eta, \VVV)$-quasirandomness.

\begin{dfn}
\label{pppqr}
A $3$-uniform hypergraph $H=(V,E)$ on $n=|V|$ vertices is \emph{$(d,\eta,\VVV)$-quasirandom} if 
for every triple of subsets $X, Y, Z\subseteq V$ the number $e_{\vvv}(X, Y, Z)$ of 
triples $(x,y,z)\in X\times Y\times Z$ 
with $xyz\in E$ satisfies
\[
\big|e_{\vvv}(X, Y, Z)-d\,|X|\,|Y|\,|Z|\big|\le \eta\,n^3\,.
\] 
\end{dfn}

The central notion for the work presented here, however, is the following stronger concept of quasi\-random hypergraphs, where we ``replace'' the two sets $Y$ and $Z$ by an arbitrary set of pairs $P$.
\begin{dfn}
\label{lpqr}
A $3$-uniform hypergraph $H=(V, E)$ on $n=|V|$ vertices is \emph{$(d, \eta,\LP)$-quasirandom}
if for every subset $X\subseteq V$ of vertices and every subset of pairs of vertices
$P\subseteq V\times V$ the number $e_{\ev}(X,P)$ of pairs $(x,(y,z))\in X\times P$
with $xyz\in E$  satisfies
\begin{equation}\label{eq:deflpqr}
	\big|e_{\ev}(X,P)-d\,|X|\,|P|\big|\leq \eta\,n^3\,.
\end{equation}
\end{dfn}

Since for any hypergraph $H=(V,E)$ and sets $X$, $Y$, $Z\subseteq V$ we have
\[
	e_{\vvv}(X,Y,Z)=e_{\ev}(X,Y\times Z)\,,
\]
it follows from these definitions that  any $(d,\eta,\EV)$-quasirandom hypergraph is also $(d,\eta,\VVV)$-quasirandom.

We also remark that the three vertices appearing in $\VVV$ 
(resp.\ the vertex and the edge depicted in~$\EV$)
symbolically represent the possible choices for the sets $X$, $Y$, $Z$ (resp.\  the set of
 vertices~$X$ and for the set of pairs $P$). 
Next we come to a notion where rather than a ``set of vertices and a set of pairs'' we consider ``two sets of pairs''.

\begin{dfn}
\label{llqr}
A $3$-uniform hypergraph $H=(V, E)$ on $n=|V|$ vertices is \emph{$(d, \eta,\QLL)$-quasirandom}
if for any two subsets of pairs $P$, $Q\subseteq V\times V$
the number $e_{\ee}(P,Q)$ of pairs of pairs $((x,y),(x,z))\in P\times Q$
with $xyz\in E$  satisfies
\[
	\big|e_{\ee}(P,Q)-d\,|\cK_{\ee}(P,Q)|\big|\leq \eta\,n^3\,,
\]
where $\cK_{\ee}(P,Q)$ denotes the set of pairs in $P\times Q$ of the form $((x,y),(x,z))$. 
\end{dfn}

Finally we will introduce the following strongest notion of quasirandomness that plays an important  r\^{o}le in the hypergraph regularity method. 

\begin{dfn}
\label{lllqr}
A $3$-uniform hypergraph $H=(V, E)$ on $n=|V|$ vertices is \emph{$(d, \eta,\LLL)$-quasirandom}
if for any three subsets $P$, $Q$, $R\subseteq V\times V$ the number 
$e_{\eee}(P,Q, R)$ of triples $((x,y),(x,z),(y,z))\in P\times Q\times R$
with $xyz\in E$  satisfies
\[
	\big|e_{\eee}(P,Q,R)-d\,|\cK_{\eee}(P,Q,R)|\big|\leq \eta\,n^3\,,
\]
where $\cK_{\eee}(P,Q,R)$ denotes the set of  triples in $P\times Q\times R$ with $((x,y),(x,z),(y,z))$.
\end{dfn}

For a symbol $\star\in\{\VVV, \LP, \QLL, \LLL\}$ we sometimes write
a hypergraph~$H$ is $\star$-quasirandom to mean that it is $(d,\eta,\star)$-quasirandom
for some positive $d$ and some small~$\eta$. More precisely, we imagine a sequence of 
$(d,\eta_n,\star)$-quasirandom hypergraphs with $\eta_n\to0$ as $n\to\infty$.
Moreover, for $\star\in\{\VVV, \LP, \QLL, \LLL\}$ we denote by $\cQ(d,\eta,\star)$ 
the class of all $(d,\eta,\star)$-quasirandom hypergraphs and 
one can observe that
\begin{equation}\label{hierarchy}
\cQ(d,\eta,\LLL)\subseteq
\cQ(d,\eta,\QLL)\subseteq
\cQ(d,\eta,\LP)\subseteq
\cQ(d,\eta,\VVV)
\end{equation}
holds for all $d\in[0,1]$ and $\eta>0$. 
We are interested in Tur\'an densities for quasirandom hypergraphs given by the following functions.
\begin{dfn} \label{def:pi} 
	Given a $3$-uniform hypergraph $F$ and a symbol $\star\in\{\VVV, \LP, \QLL, \LLL\}$ we set
\begin{multline*}
	\pi_\star(F)=\sup\bigl\{d\in[0,1]\colon \text{for every $\eta>0$ and $n\in \NN$ there exists}\\
		\text{an $F$-free $3$-uniform hypergraph $H\in\cQ(d,\eta, \star)$ with $|V(H)|\geq n$}\bigr\}\,.
\end{multline*}
\end{dfn} 
Due to the inclusions~\eqref{hierarchy} we have
\begin{equation}\label{qr-order}
\pivvv(F)\ge\pilp(F)\ge\pill(F)\ge\pilll(F)
\end{equation}
for any $3$-uniform hypergraph $F$. The last among those four parameters is trivial 
because we have $\pilll(F)=0$ for any hypergraph $F$. This follows directly from the results in~\cite{KRS02}
(alternatively, it can be easily deduced by the regularity method for hypergraphs via a combined application of
Theorems~\ref{thm:TuRL} and~\ref{thm:EL}). 

Moreover, when~$F$ is tripartite or linear it can be shown that all four of these quantities vanish 
(in fact, a full characterisation of this event is going to appear in \cite{RRS-zero}).
For \mbox{$F=K_4^{(3)-}$}, the $3$-uniform hypergraph consisting of three hyperedges on four vertices,
it is known that 
\[
	\pivvv(K_4^{(3)-})=\pilp(K_4^{(3)-})=1/4
	\qqand
	\pill(K_4^{(3)-})=0\,.
\]
In fact, Glebov, Kr{\'a}{\soft{l}}, and Volec established $\pivvv(K_4^{(3)-})\le 1/4$
in~\cite{GKV} (see also~\cite{RRS-a} for an alternative proof).
On the other hand, one can check that 
the hypergraph corresponding to the cyclically oriented triangles of a
random tournament, which provides $\pivvv(K_4^{(3)-})\ge 1/4$, 
is also $\LP$-quasirandom, and by \eqref{qr-order}
we get
\[
	\frac14\le \pilp(K_4^{(3)-})\le \pivvv(K_4^{(3)-})\le \frac14\,,
\]
which establishes our first claim. 
The second identity  follows from $\pill(K^{(3)}_4)=0$, which will appear in \cite{RRS-d}. 

The next case one might wish to study is the tetrahedron. The following random
construction from~\cite{Ro86} was used to show that $\pivvv(K_4^{(3)})\ge 1/2$
and Erd\H os~\cite{Er90} suggested that this might be best possible. 
\begin{exmp}\label{VR-graph}
Given any map $\varphi\colon\binom{[n]}{2}\to\{\text{red}, \text{green}\}$ we define the $3$-uniform 
hypergraph~$H_\varphi$ with vertex set $[n]=\{1,\dots,n\}$ by putting a triple $ijk$ with $i<j<k$ into~$E(H_\varphi)$ if and only if the colours of $ij$ and $ik$ differ. 

Irrespective of the choice of the colouring $\varphi$, the hypergraph $H_\varphi$ contains no tetrahedra: for if $a$, $b$, $c$, and $d$ are any four distinct vertices, say with $a=\min\{a,b,c,d\}$, then it is impossible for all three of the pairs $ab$, $ac$, and $ad$ to have distinct colours, whence not all three of the triples $abc$, $abd$, and $acd$ can be hyperedges of $H_\varphi$.  

It was noticed in \cite{Ro86} that if the colouring $\varphi$ is chosen uniformly at random, then for any $\eta>0$ the hypergraph $H_\varphi$ is with high probability $(1/2, \eta, \VVV)$-quasirandom as $n$ tends to infinity. This is easily checked using standard tail estimates for binomial distributions and similar arguments show that we may replace $\VVV$ by $\LP$ in this observation. In other words, this example shows that
\[
\pilp(K^{(3)}_4)\ge\tfrac12\,.
\]
holds.	
\end{exmp}
Our main contribution here shows that the construction given in Example~\ref{VR-graph} 
is best possible.        
\begin{thm}[Main result]\label{K4lp}
	For every $\eps>0$ there exist an $\eta>0$ and an integer $n_0$ such that
	every 3-uniform $(\tfrac12+\eps, \eta, \LP)$-quasirandom hypergraph $H$
	with  at least $n_0$ vertices contains a tetrahedron. 
\end{thm}
The proof of Theorem~\ref{K4lp} will be based on the regularity method for $3$-uniform 
hypergraphs which is summarised to the necessary extent in the following section. The details of the 
proof of Theorem~\ref{K4lp} appear in Section~\ref{sec:K4}. We close with a few remarks
on extremal problems involving the tetrahedron for related notions of quasirandomness in Section~\ref{sec:conc}.

\section{Hypergraph regularity method}\label{sec:regmethod}
A key tool in the proof of Theorem~\ref{K4lp} is the regularity lemma for $3$-uniform hypergraphs. 
We follow the approach from~\cites{RoSchRL,RoSchCL} combined with the results from~\cite{Gow06} and~\cite{NPRS09}
and below we introduce the necessary notation.

For two disjoint sets $X$ and $Y$ we denote by $K(X,Y)$ the complete bipartite graph with that vertex partition.
We say a bipartite graph $P=(X\dcup Y,E)$ is \emph{$(\delta_2, d_2)$-regular} if for all subsets 
$X'\subseteq X$ and $Y'\subseteq Y$ we have 
\[
	\big|e(X',Y')-d_2|X'||Y'|\big|\leq \delta_2 |X||Y|\,,
\]
where $e(X',Y')$ denotes the number of edges of $P$ with one vertex in $X'$ and one vertex in~$Y'$.
Moreover, for $k\geq 2$ we say a $k$-partite graph $P=(X_1\dcup \dots\dcup X_k,E)$ is $(\delta_2, d_2)$-regular, 
if all its  $\binom{k}{2}$ naturally 
induced bipartite subgraphs $P[X_i,X_j]$ are $(\delta_2, d_2)$-regular. 
For a tripartite graph $P=(X\dcup Y\dcup Z,E)$
we denote by $\cK_3(P)$ the triples of vertices spanning a triangle in~$P$, i.e., 
\[
	\cK_3(P)=\big\{\{x,y,z\}\subseteq X\cup Y\cup Z\colon xy, xz, yz\in E\big\}\,.
\]
If the tripartite graph $P$ is $(\delta_2, d_2)$-regular, then the so-called \emph{triangle counting lemma}
implies
\begin{equation}
	\label{eq:TCL}
		|\cK_3(P)|\leq d_2^3|X||Y||Z|+3\delta_2|X||Y||Z|\,.
\end{equation}

We say a $3$-uniform hypergraph $H=(V,E_H)$ is regular w.r.t.\ a tripartite graph $P$ if it matches 
approximately
the same proportion of triangles for every subgraph $Q\subseteq P$. This we make precise in the following definition.

\begin{dfn}
\label{def:reg}
A $3$-uniform hypergraph $H=(V,E_H)$ is \emph{$(\delta_3,d_3)$-regular w.r.t.\ 
a tripartite graph $P=(X\dcup Y\dcup Z,E_P)$} 
with $V\supseteq  X\cup Y\cup Z$ if for every tripartite subgraph $Q\subseteq P$ we have 
\[
	\big||E_H\cap\cK_3(Q)|-d_3|\cK_3(Q)|\big|\leq \delta_3|\cK_3(P)|\,.
\]
Moreover, we simply say \emph{$H$ is $\delta_3$-regular w.r.t.\ $P$}, if it is $(\delta_3,d_3)$-regular for some $d_3\geq 0$.
We also define the \emph{relative density} of $H$ w.r.t.\ $P$ by
\[
	d(H|P)=\frac{|E_H\cap\cK_3(P)|}{|\cK_3(P)|}\,,
\]
where we use the convention $d(H|P)=0$ if $\cK_3(P)=\emptyset$. If $H$ is not $\delta_3$-regular w.r.t.\ $P$, then we simply refer to it as \emph{$\delta_3$-irregular}.
\end{dfn}

The regularity lemma for $3$-uniform hypergraphs, introduced by Frankl and R\"odl in~\cite{FR}, provides for 
every hypergraph $H$ a partition of its vertex set and a partition of the edge sets of the complete bipartite 
graphs induced by the vertex partition such that for appropriate constants $\delta_3$, $\delta_2$, and $d_2$ 
\begin{enumerate}[label=\nlabel]
	\item the bipartite graphs given by the partitions are $(\delta_2,d_2)$-regular and
	\item $H$ is $\delta_3$-regular for ``most'' tripartite graphs $P$ given by the partition.
\end{enumerate}
In many proofs based on the regularity method it is
convenient to ``clean'' the regular partition provided by the regularity lemma. In particular, 
we shall disregard hyperedges of~$H$ that belong to $\cK_3(P)$ where $H$ is not $\delta_3$-regular or 
where $d(H|P)$ is very small. These properties are rendered in the following somewhat standard
corollary of the regularity lemma.

\begin{thm}
	\label{thm:TuRL}
	For every $d_3>0$, $\delta_3>0$ and  $m\in\NN$, and every function $\delta_2\colon \NN \to (0,1]$,
	there exist integers~$T_0$ and $n_0$ such that for every $n\geq n_0$
	and every $n$-vertex $3$-uniform hypergraph $H=(V,E)$ the following holds.
	
	There exists a subhypergraph $\hat H=(\hat V,\hat E)\subseteq H$, an integer $\l\leq T_0$,
	a vertex partition $V_1\dcup\dots\dcup V_m=\hat V$, 
	and for all $1\leq i<j\leq m$ there exists 
	a partition 
	\[
		\cP^{ij}=\{P^{ij}_\alpha=(V_i\dcup V_j,E^{ij}_\alpha)\colon 1\leq \alpha \leq \l\}
	\] 
	of $K(V_i,V_j)$ satisfying the following properties
	\begin{enumerate}[label=\rmlabel]
		\item\label{TuRL:1} $|V_1|=\dots=|V_m|\geq (1-\delta_3)n/T_0$,
		\item\label{TuRL:2} for every $1\leq i<j\leq m$ and $\alpha\in [\l]$ the bipartite graph $P^{ij}_\alpha$ is $(\delta_2(\l),1/\l)$-regular,
		\item $\hat H$ is $\delta_3$-regular w.r.t.\ $P^{ijk}_{\alpha\beta\gamma}$
			for all tripartite graphs (which will be later referred to as triads)
			\begin{equation}\label{eq:triad}
				P^{ijk}_{\alpha\beta\gamma}=P^{ij}_\alpha\dcup P^{ik}_\beta\dcup P^{jk}_\gamma=(V_i\dcup V_j\dcup V_k, E^{ij}_\alpha\dcup E^{ik}_{\beta}\dcup E^{jk}_{\gamma})\,,
			\end{equation}
			with $1\leq i<j<k\leq m$ and $\alpha$, $\beta$, $\gamma\in[\l]$, where the relative density $d(H|P)$ is either $0$ or
			$d(H|P)\geq d_3$, and			
		\item\label{TuRL:4} for every $1\leq i<j<k\leq m$ 												there are at most $\delta_3\,\ell^3$ triples $(\alpha, \beta, \gamma)\in[\ell]^3$
			such that~$H$ is $\delta_3$-irregular with respect to the triad 
			$P^{ijk}_{\alpha\beta\gamma}$.
	\end{enumerate}
\end{thm}
The standard proof of Theorem~\ref{thm:TuRL} based on a refined version of 
the regularity lemma from~\cite{RoSchRL}*{Theorem~2.3}
can be found in~\cite{RRS-a}*{Corollary 3.3}. Actually the statement there differs from the one given here in the final clause, but the proof from~\cite{RRS-a} shows the present version as well. (In fact, the new version of \ref{TuRL:4} is explicitly stated as clause (a) in the definition of the hypergraph $R$ in~\cite{RRS-a}*{Proof of Corollary 3.3}.)  

We shall use a so-called \emph{counting/embedding lemma}, which allows us to embed hypergraphs of fixed isomorphism type 
into appropriate and sufficiently regular and dense triads of the partition provided by Theorem~\ref{thm:TuRL}. We state 
a simplified version of the embedding lemma for tetrahedra only.
The following statement is a direct consequence of~\cite{NPRS09}*{Corollary~2.3}.

\begin{thm}[Embedding Lemma for $K_4^{(3)}$]
	\label{thm:EL}
	For every $d_3>0$ there exist $\delta_3>0$ and functions 
	$\delta_2\colon \NN\to(0,1]$ and  $N\colon \NN\to\NN$
	such that the following holds for every $\l\in\NN$.

	Suppose $P=(V_1\dcup\dots\dcup V_4, E_P)$ is a $(\delta_2(\l),\frac{1}{\l})$-regular, $4$-partite graph
	with vertex classes satisfying  $|V_1|=\dots=|V_4|\geq N(\l)$ and suppose $H$ is a $4$-partite, $3$-uniform hypergraph
	such that for all $1\leq i< j< k\leq 4$ we have
	\begin{enumerate}[label=\alabel]
		\item\label{EL:a} $H$ is $\delta_3$-regular w.r.t.\ to the tripartite graph $P[V_i\dcup V_j\dcup V_k]$ and 
		\item\label{EL:b} $d(H|P[V_i\dcup V_j\dcup V_k])\geq d_3$,
	\end{enumerate} 
	then $H$ contains a copy of $K_4^{(3)}$.
\end{thm}
In an application of Theorem~\ref{thm:EL} the tripartite graphs $P[V_i\dcup V_j\dcup V_k]$ in~\ref{EL:a} and~\ref{EL:b}
will be given by triads  $P^{ijk}_{\alpha\beta\gamma}$ from the partition given by Theorem~\ref{thm:TuRL}.
For the proof of Theorem~\ref{K4lp} we consider a $\LP$-quasirandom hypergraph~$H$ of density $1/2+\eps$.
We will apply the regularity lemma in the form of Theorem~\ref{thm:TuRL} to~$H$. The main part of the proof
concerns the appropriate selection of dense and regular triads, that are ready for an application of the embedding lemma
for $K^{(3)}_4$. This will be the focus in Section~\ref{sec:K4}.

\section{Embedding tetrahedra}
\label{sec:K4}

In this section we deduce Theorem~\ref{K4lp}. The proof will be based on the regularity method for hypergraphs in the form of 
Theorem~\ref{thm:TuRL} and the embedding lemma (Theorem~\ref{thm:EL}). 
Moreover, it relies on the following lemma (see Lemma~\ref{K4modif} below),
which locates in a sufficiently regular partition of a $\LP$-quasirandom hypergraph with density 
$>1/2$ a collection of triads that are ready for an application of the embedding lemma for~$K_4^{(3)}$.
\begin{lemma}\label{K4modif}
For every $\eps>0$ there exist $\delta>0$ and an integer $m$ such  that
the following holds.
If $\cA$ is an $\binom{m}{2}$-partite $3$-uniform hypergraph with 
\begin{enumerate}[label=\rmlabel]
\item nonempty vertex classes $\cP^{ij}$ for 
$1\le i<j\le m$ such that
\item \label{it:K4modif-b}for each triple $ijk\in\binom{[m]}{3}$ the 
	restriction $\cA^{ijk}$ of $\cA$ to $\cP^{ij}\cup \cP^{ik}\cup \cP^{jk}$ 
	has the property that all but at most $\delta|\cP^{ab}|$ vertices of $\cP^{ab}$ are contained in 
	at least  $(1/2+\eps)|\cP^{ac}||\cP^{bc}|$ hyperedges of $\cA^{ijk}$ for all $\{a,b,c\}=\{i,j,k\}$, 
\end{enumerate}
then there are four distinct indices $i_1$, $i_2$, $i_3$, and $i_4$ from~$[m]$ together 
with six vertices $P^{ab}\in \cP^{i_ai_b}$ for $1\le a<b\le4$ such that all four $P^{12}P^{13}P^{23}$,
$P^{12}P^{14}P^{24}$, $P^{13}P^{14}P^{34}$, and~$P^{23}P^{24}P^{34}$ are triples of $\cA$. 
\end{lemma}
\begin{proof}
For a given  $\eps>0$ we set
\[
\delta=\frac{\eps}{4} \qqand m=3+2^{1/\delta^3}\,.
\]
We will find the desired configuration with $i_1=1$ and $i_2=2$. 
The argument splits into three steps. In the first step 
we select the indices $i_3$ and $i_4$. 
Then, we select the six vertices  $P^{i_ai_b}$ with $1\leq a<b\leq 4$. In the second step we
shall fix the three vertices $P^{i_ai_4}$ with $a=1,2,3$ and in the third step
we fix the remaining three vertices.

\smallskip
\noindent
\textbf{Step 1: Selecting $i_3$ and $i_4$.} 
We commence by assigning a colour to each integer between~$3$ and $m$, the idea being that the colour of an index $i\in[3,m]$ encodes the sizes of holes (independent sets) in $\cA^{12i}$. 
More precisely, for positive integers $p$, $q$, and $r$ we say, 
that~$\cA^{12i}$ has a \emph{$(p,q, r)$-hole} if there are three sets $I^{12}\subseteq \cP^{12}$, \
$I^{1i}\subseteq \cP^{1i}$, and $I^{2i}\subseteq \cP^{2i}$ 
with 
\[
	|I^{12}|\ge p\cdot\delta\,|\cP^{12}|\,,\quad
	|I^{1i}|\ge q\cdot\delta\,|\cP^{1i}|\,, \qand 
	|I^{2i}|\ge r\cdot\delta\,|\cP^{2i}|
\]
such that $I^{12}\cup I^{1i}\cup I^{2i}$ is independent in $\cA^{12i}$. 
Evidently, such a hole can only exist if $p,q, r\le \delta^{-1}$. 

Let $\Xi_i$ be the set of all integer triples $(p,q,r)\in\bigl[1, \delta^{-1}\bigr]^3$ 
for which $\cA^{12i}$ contains a $(p,q,r)$-hole.
We think of the set $\Xi_i$ as a colour that has been attributed to the index~$i\in[3,m]$
and obviously there are at most $2^{1/\delta^3}$ possible colours.
Since $m-2$ exceeds the number of possible colours, the pigeonhole principle tells us that there exist two distinct integers $i_3$ and~$i_4$ between $3$~and~$m$ of the same colour. For the rest of the proof
only the parts of $\cA$ accessible via the indices $1$, $2$, $i_3$, and $i_4$ are relevant
and so without loss of generality we may henceforth assume $i_3=3$, $i_4=4$, and 
\begin{equation}\label{eq:Xi}
	\Xi_3=\Xi_4\,.
\end{equation}

\smallskip
\noindent
\textbf{Step 2: Choosing $P^{14}$, $P^{24}$, and $P^{34}$.} 
For any three distinct indices $i$, $j$, and $k$ we denote the set of all vertices from $\cP^{ij}$ whose degree in $\cA^{ijk}$ is less than $\bl\tfrac12+\eps\br|\cP^{ik}|\,|\cP^{jk}|$ 
by $X_k^{ij}$. In view of assumption~\ref{it:K4modif-b} of Lemma~\ref{K4modif}
we have 
\begin{equation}\label{eq:X}
	|X^{ij}_k|\leq\delta\,|\cP^{ij}|\,.
\end{equation}
 Given two vertices $P$ and $P'$ of $\cA$, 
we write $\codeg(P, P')$ for their \emph{codegree}, i.e., for the size $|N(P,P')|$ of their 
\emph{common neighbourhood} 
\[
	N(P,P')=\bigl\{P''\in V(\cA)\colon PP'P''\in E(\cA)\bigr\}\,.
\]
It follows from the structure of $\cA$ that $N(P,P')$ is not empty only when $P\in\cP^{ij}$ and 
$P'\in \cP^{ik}$ for some distinct indices $i$, $j,$ and $k$ and in this case we have  
 $N(P,P')\subseteq \cP^{jk}$. 

We fix $P^{14}$ and $P^{24}$ by selecting a pair 
\[
	(P^{14},P^{24})\in (\cP^{14}\setminus X^{14}_3)\times (\cP^{24}\setminus X^{24}_3)
\]
with maximum codegree in $\cA$. Let $p$ be the largest integer such that 
\begin{equation}\label{eq:1424a}
	\codeg(P^{14},P^{24})\geq p\cdot \delta|\cP^{12}|\,.
\end{equation}
It follows from assumption~\ref{it:K4modif-b} that the average codegree among all pairs 
in $\cP^{14}\times \cP^{24}$ is at least 
\[
	(1-\delta)(1/2+\eps)|\cP^{12}|
\] 
and, since $X^{14}_3$ and $X^{24}_3$ are small, a similar lower bound holds for the average (and hence the maximum) codegree in $(\cP^{14}\setminus X^{14}_3)\times (\cP^{24}\setminus X^{24}_3)$. In fact, the number of hyperedges 
in $\cA^{124}$ avoiding vertices in $X^{14}_3$ and $X^{24}_3$ is at least
\begin{multline*}
	(1/2+\eps)|\cP^{12}|\cdot (1-\delta)|\cP^{14}||\cP^{24}|
		-|\cP^{12}||X^{14}_3||\cP^{24}|-|\cP^{12}||\cP^{14}||X^{24}_3|\\
	\geq
	\big((1/2+\eps)(1-\delta)-2\delta\big)|\cP^{12}||\cP^{14}||\cP^{24}|\,.
\end{multline*}
Note that we may assume that $\eps\leq 1/2$, as otherwise the lemma is void.
Consequently, the average codegree of the pairs in 
$(\cP^{14}\setminus X^{14}_3)\times (\cP^{24}\setminus X^{24}_3)$ is at least $(1/2+\eps-3\delta)|\cP^{12}|$.
Owing to the maximal choice of~$p$ in~\eqref{eq:1424a}, this yields
\begin{equation}\label{eq:1424b}
	p\,\delta\geq \tfrac{1}{2}+\eps-4\delta\geq \tfrac{1}{2}\,.
\end{equation}

Having selected $P^{14}$ and $P^{24}$ we now select $P^{34}$. 
Due to $P^{14}\not\in X^{14}_3$ and $P^{24}\not\in X^{24}_3$ we have
\[
	\sum_{P\in \cP^{34}}\codeg(P^{14},P)\ge(\tfrac12+\eps)|\cP^{13}|\,|\cP^{34}|
\]
as well as 
\[ 
	\sum_{P\in \cP^{34}}\codeg(P^{24},P)\ge(\tfrac12+\eps)|\cP^{23}|\,|\cP^{34}|\,,
\]
whence
\[
\sum_{P\in \cP^{34}}\left(\frac{\codeg(P^{14},P)}{|\cP^{13}|}
	+\frac{\codeg(P^{24},P)}{|\cP^{23}|}\right)\ge(1+2\eps)|\cP^{34}|\,.
\]
For this reason, we may choose a vertex $P^{34}\in \cP^{34}$ 
in such a way that 
\begin{equation}\label{eq:34a}
	\frac{\codeg(P^{14},P^{34})}{|\cP^{13}|}+\frac{\codeg(P^{24},P^{34})}{|\cP^{23}|}\ge 1+2\eps\,.
\end{equation}
Because of $N(P^{14},P^{34})\subseteq \cP^{13}$ and $N(P^{24},P^{34})\subseteq \cP^{23}$ it follows that 
\begin{equation}\label{eq:34b}
	|N(P^{14},P^{34})|\ge 2\,\eps\,|\cP^{13}|\qand
	|N(P^{24},P^{34})|\ge 2\,\eps\,|\cP^{23}|\,.
\end{equation}
This concludes the selection of $P^{14}$, $P^{24}$, and $P^{34}$, which by~\eqref{eq:1424a}, 
\eqref{eq:1424b}, and \eqref{eq:34b} guarantees that all three possible 
codegrees of these vertices are reasonable large. It is left to select~$P^{12}$,~$P^{13}$, and $P^{23}$. These three vertices have to form a hyperedge in 
$\cA$ and each of them must be chosen from the common neighbourhood of two vertices chosen already, i.e.,
we have to make sure that there is a hyperedge  $P^{12}P^{13}P^{23}$ of $\cA$ with 
$P^{12}\in N(P^{14},P^{24})$, $P^{13}\in N(P^{14},P^{34})$, and $P^{23}\in N(P^{24},P^{34})$.
This is the content of the last step.

\smallskip
\noindent
\textbf{Step 3: Choosing $P^{12}$, $P^{13}$, and $P^{23}$.} 
As mentioned above it suffices to find a hyperedge with each vertex coming from 
one of the common neighbourhoods
\[
	I^{12} = N(P^{14},P^{24})\,,\quad
	I^{13} = N(P^{14},P^{34})\,,
	\qand 
	I^{23} = N(P^{24},P^{34})\,,
\] 
since this would give rise to a choice of $P^{12}$, $P^{13}$, and $P^{23}$ with the desired properties.

Suppose to the contrary that $\cA^{123}[I^{12}, I^{13},I^{23}]$ is independent, i.e., 
it gives rise to a $(p,q,r)$-hole in $\cA^{123}$, where $q$ and $r$ are the largest integers such that
\[
	|I^{13}|\geq q\cdot\delta|\cP^{13}|
	\qqand
	|I^{23}|\geq r\cdot\delta|\cP^{23}|
\] 
(for the definition of $p$ see~\eqref{eq:1424a}). Owing to~\eqref{eq:34a}  and the maximality 
of $q$ and $r$ we have
\[
	(q+1)\delta+(r+1)\delta> 1+2\eps\,.
\] 
Since $q\delta$ and $r\delta$ are bounded by $1$ we also have
\begin{equation}\label{eq:lbr}
	q\delta> 2(\eps-\delta)
	\qqand
	r\delta> 2(\eps-\delta)\,.
\end{equation}
Without loss of generality we may assume $q\geq r$ and since $\delta<\eps$ we then have
\begin{equation}\label{eq:lbq}
	q\delta> 1/2\,.
\end{equation}

Since the sets $I^{12}$, $I^{13}$, and $I^{23}$ form a $(p,q,r)$-hole in $\cA^{123}$, we have
 $(p,q,r)\in\Xi_3$ and owing 
to~\eqref{eq:Xi} we know that $(p,q,r)$ is also in $\Xi_4$, i.e., 
$\cA^{124}$ also contains a $(p,q,r)$-hole. This gives rise to an independent set in $\cA^{124}$ formed by
$J^{12}\subseteq \cP^{12}$, $J^{14}\subseteq \cP^{14}$, and $J^{24}\subseteq \cP^{24}$ with 
\[
	|J^{12}|\geq p\cdot\delta|\cP^{12}|\,,\qquad
	|J^{14}|\geq q\cdot\delta|\cP^{14}|\,,
	\qqand
	|J^{24}|\geq r\cdot\delta|\cP^{24}|\,.
\]
We will
use the fact that the chosen pair $(P^{14},P^{24})$ maximises the codegree in $\cA^{124}$
over all pairs in $(\cP^{14}\setminus X^{14}_3)\times (\cP^{24}\setminus X^{24}_3)$.
Owing to the maximal choice of $p$ in~\eqref{eq:1424a} we have 
\[
	\codeg(P^{14},P^{24})<(p+1)\cdot \delta|\cP^{12}|\,.
\] 
We consider the set
\[
	J^{24}_0=J^{24}\setminus (X^{24}_3\cup X^{24}_1)\,.
\]
We will arrive at a contradiction by considering a vertex from $J^{24}_0$.
It follows from~\eqref{eq:lbr} and~\eqref{eq:X} that
\[
	|J^{24}_0|\geq |J^{24}|-|X^{24}_3|-|X^{24}_1|\geq 2(\eps-2\delta)|\cP^{24}|>0\,.
\]
Therefore, there exists some vertex $Q^{24}\in J^{24}_0$, which we fix for the rest of the proof. We estimate the degree $\deg(Q^{24})$ 
of $Q^{24}$ in $\cA^{124}$ in two ways. Since $Q^{24}\not\in X^{24}_1$ we have
\begin{equation}\label{eq:lb}
	\deg(Q^{24})\geq (1/2+\eps)|\cP^{12}||\cP^{14}|\,.
\end{equation}
On the other hand, we write $\deg(Q^{24})$ as the sum of all codegrees of $Q^{24}$ with a vertex from~$\cP^{14}$, i.e.,
\[
	\deg(Q^{24})=\sum_{Q\in\cP^{14}}\codeg(Q,Q^{24})
\]
and consider three cases depending on $Q$. If $Q\in J^{14}$, then all common neighbours of 
$Q^{24}$ and~$Q$ must lie outside $J^{12}$, as $J^{12}$, $J^{14}$, and $J^{24}$ form a hole 
in $\cA^{124}$. In particular, in this case we have
\[
	\codeg(Q,Q^{24})\leq |\cP^{12}|-|J^{12}|\leq (1-p\delta)|\cP^{12}|\,.
\]

For the second case, we consider $Q\in X^{14}_3$ and we use the trivial upper bound 
\[
	\codeg(Q,Q^{24})\leq |\cP^{12}|\,.
\]
However, due to~\eqref{eq:X} we know that this will only contribute little to $\deg(Q^{24})$.

In the remaining case we have $Q\not\in J^{14}$ and $Q\not\in X^{14}_3$. Then we have
\[(Q,Q^{24})\in (\cP^{14}\setminus X^{14}_3)\times (\cP^{24}\setminus X^{24}_3)\] and by the maximal choice of 
$(P^{14},P^{24})$ we infer
\[
	\codeg(Q,Q^{24})\leq \codeg(P^{14},P^{24}) < (p+1)\delta|\cP^{12}|\,.
\]

Putting the three cases together, we obtain
\[
	\deg(Q^{24}) \leq |J^{14}|\cdot (1-p\delta)|\cP^{12}| 
		+ |X^{14}_3|\cdot |\cP^{12}| 
		+ (|\cP^{14}|-|J^{14}|)\cdot(p+1)\delta|\cP^{12}|
\]
Let $x$, $y\in\RR$ be given by $x=p\delta$ and 
$y=|J^{14}|/|\cP^{14}|$. Recalling~\eqref{eq:1424b} and~\eqref{eq:lbq} we note that
$x$, $y\geq 1/2$ and we can rewrite the last inequality as
\[
	\frac{\deg(Q^{24})}{|\cP^{12}||\cP^{14}|}
	\leq 
	y(1-x) + \delta + (1-y)x+\delta\,.
\]
Comparing this with~\eqref{eq:lb} we arrive at
\[
	\frac{1}{2}+\eps-2\delta \leq y(1-x)+(1-y)x\,,
\]
which due to $x$, $y\geq 1/2$ leads to the contradiction
\[
	0\leq\tfrac{1}{2}(2x-1)(2y-1)\leq 2\delta-\eps<0
\]
and concludes the proof of Lemma~\ref{K4modif}.
\end{proof}

The proof of the main result follows by a combined application of the regularity method 
for hypergraphs and Lemma~\ref{K4modif}.

\begin{proof}[Proof of Theorem~\ref{K4lp}]
Given $\eps>0$ we have to find appropriate $\eta>0$ and $n_0\in\NN$. For this purpose we start by choosing some 
auxiliary constants $\delta_3$, $d_3$, $\delta$, and $m$ obeying the hierarchy 
\begin{equation}\label{eq:hier}
\delta_3\ll d_3, \delta, m^{-1}\ll\eps\,.
\end{equation}
For these choices of $\delta_3$ and $d_3$ and $F=K_4^{(3)}$ we appeal to Theorem~\ref{thm:EL} and 
obtain $\delta_2\colon\NN\to\NN$ and $N\colon \NN\to\NN$. Without loss of generality we may assume 
that for all $\ell\in\NN$ we have
\[
\delta_2(\ell)\ll\ell^{-1}, \eps\,.
\]
Applying Theorem~\ref{thm:TuRL} to $d_3$, $\delta_3$, $m$, and $\delta_2$ we 
get two integers $T_0$ and $n'_0$. Now we claim that any 
\[
\eta\ll T_0^{-1} \qquad \text{ and } \qquad n_0\gg \max\{T_0\cdot N(T_0), n'_0\}
\]
are as desired. 

To justify this, we let any $(1/2+\eps, \eta, \EV)$-quasirandom hypergraph $H$ on $n\ge n_0$ vertices be given. Since $n\ge n'_0$ holds as well, we may apply Theorem~\ref{thm:TuRL}, thus getting a subhypergraph $\hat H\subseteq H$ with vertex partition $\hat V=V_1\dcup\dots\dcup V_m$ and edge partitions
$\cP^{ij}=\{P^{ij}_{\alpha}\colon \alpha\in[\l]\}$ of $K(V_i,V_j)$ for $1\le i<j\le m$.      

In view of the embedding lemma (Theorem~\ref{thm:EL}) the task that remains to be done is now reduced to the task of locating four vertex classes
$V_{i_1},\dots,V_{i_4}$ with $i_1< i_2< i_3<i_4$ and six bipartite graphs $P^{ab}\in \cP^{i_ai_b}$ for $1\le a<b\le4$
from the regular partition, such that all triads 
\[
	P^{abc}=P^{ab}\dcup P^{ac}\dcup P^{bc}
\] 
with $1\le a<b<c\le4$
are \emph{dense} and \emph{regular}, i.e., $d(H|P^{abc})\geq d_3$ and $H$ is $\delta_3$-regular w.r.t.~$P^{abc}$. For this purpose we reformulate our current situation in terms of a ``reduced hypergraph'' $\cA$. 

The reduced hypergraph $\cA$ is going to be a $3$-uniform and $\binom{m}{2}$-partite hypergraph with vertex classes~$\cP^{ij}$ of size $\l$ for $1\le i<j\le m$. Among its $\binom{\binom{m}{2}}{3}$ naturally induced tripartite $3$-uniform subhypergraphs only $\binom{m}{3}$ many are inhabited by hyperedges: these are the hypergraphs $\cA^{ijk}$ with vertex classes~$\cP^{ij}$, $\cP^{ik}$, and $\cP^{jk}$ for $1\le i<j<k\le m$. They are defined to have precisely those hyperedges 
$P^{ij}P^{ik}P^{jk}$ with
$P^{ij}\in\cP^{ij}$, $P^{ik}\in\cP^{ik}$, and $P^{jk}\in\cP^{jk}$ for which the triad
\[
(V_i\dcup V_j\dcup V_k, P^{ij}\dcup P^{ik}\dcup P^{jk})
\]
has $\hat H$-density at least $d_3$. To see that Lemma~\ref{K4modif} is applicable (with $\eps/2$ instead of $\eps$), 
it is enough to verify that given $1\le i<j<k\le m$ the following is true. There are at most~$\delta\,|\cP^{ij}|$ vertices $P^{ij}\in \cP^{ij}$ whose degree in $\cA^{ijk}$ is smaller than $(1/2+\eps/2)\ell^2$, and similarly for $\cP^{ik}$ and~$\cP^{jk}$.     

Since the proofs of all three of these statements are the same, we just deal with $\cP^{ij}$ in the sequel. Let $X^{ij}_{k}$ denote the set of all those $P^{ij}\in\cP^{ij}$ for which there are more than~$\delta\,\ell^2$ pairs $(P^{ik}, P^{jk})\in\cP^{ik}\times\cP^{jk}$ such that $H$ is $\delta_3$-irregular with respect to the triad $\bl V_i\dcup V_j\dcup V_k, P^{ij}\dcup P^{ik}\dcup P^{jk}\br$. Consequently the total number of triads involving~$V_i$,~$V_j$, and $V_k$ with respect to which $H$ is $\delta_3$-irregular is on the one hand at least $\delta\,\ell^2\,|X^{ij}_{k}|$. On the other hand it is at most $\delta_3\ell^3$ by clause \ref{TuRL:4} of Theorem~\ref{thm:TuRL}. Consequently we have $|X^{ij}_k|\le\delta_3\,\ell/\delta\leq \delta\l$ (by the hierarchy given in~\eqref{eq:hier}).  It suffices to check that any $P^{ij}\in\cP^{ij}\setminus X^{ij}_k$ belongs to at least $(1/2+\eps/2)\ell^2$ hyperedges of $\cA^{ijk}$.

To verify this, we fix any $P^{ij}\in \cP^{ij}\setminus X^{ij}_{k}$ for the remainder of the argument. 
Let us apply the $(1/2+\eps, \eta, \EV)$-quasirandomness of $H$ to $V_k$ and the set of pairs
\[
	Q^{ij}=\{(x,y)\in V_i\times V_j\colon xy\in E(P^{ij})\}
\]
in the r\^ole of $X$ and $P$ of Definition~\ref{lpqr}. Concerning the number $e_{\ev}(V_k,Q^{ij})$ of pairs
$(v,(x,y))\in V_k\times Q^{ij}$ with $vxy\in E(H)$ this tells us 
\[
e_{\ev}(V_k,Q^{ij})\ge \bl\tfrac12+\eps\br |V_k|\,|E(P^{ij})|-\eta \cdot n^3\,.
\]
Set $M=|V_i|=|V_j|=|V_k|$. Since 
\[
|E(P^{ij})|\ge \left(\frac1\ell-\delta_2(\ell)\right)M^2\]
follows from \ref{TuRL:2}, we get
\[
e_{\ev}(V_k,Q^{ij})\ge \left(\frac12+\eps\right)\left(\frac1\ell-\delta_2(\ell)\right)M^3-\eta\,n^3\,.
\]
As we have $M\ge \tfrac{n}{2T_0}$ by \ref{TuRL:1}, 
the hierarchy imposed on $\eta$ leads to  
\[
e_{\ev}(V_k,Q^{ij})\ge \left(\frac 12+\frac{9\eps}{10}\right)\cdot\frac{M^3}{\ell}\,.
\]
On the other hand, we have
\begin{equation}\label{T-sum}
	e_{\ev}(V_k,Q^{ij})=\sum_{(P^{ik},P^{jk})\in \cP^{ik}\times\cP^{jk}}\big|E(H)\cap\cK_3(P^{ij}\cup P^{ik}\cup P^{jk})\big|\,.
\end{equation}
The terms corresponding to triads $\bl V_i\dcup V_j\dcup V_k, P^{ij}\dcup P^{ik}\dcup P^{jk}\br$ with respect to which~$H$ has density at most $d_3$ contribute at most 
$d_3(\ell^{-3}+3\delta_2(\ell))M^3\ell^2$ (see~\eqref{eq:TCL})
 and by $\delta_2(\ell)\ll\ell^{-1}$ this is at most $2d_3M^3/\ell$. 

Further, by $P^{ij}\not\in X^{ij}_k$ there are at most $\delta\,\ell^2$ terms on the right hand side of \eqref{T-sum} corresponding to $\delta_3$-irregular triads, and each of them contributes, for the same reason as above, at most $\tfrac {2M^3}{\ell^3}$ to the right hand side of \eqref{T-sum}.
 
The remaining terms from \eqref{T-sum} satisfy $P^{ij}P^{ik}P^{jk}\in E(\cA^{ijk})$ and 
each of them contributes at most $\bl 1+\tfrac\eps 5\br\cdot\tfrac{M^3}{\ell^3}$. So if $\deg(P^{ij})$ denotes the degree of $P^{ij}$ in $\cA^{ijk}$ we arrive at
\[
e_{\ev}(V_k,Q^{ij})\le \left(\left(1+\frac\eps 5\right)\frac{\deg(P^{ij})}{\ell^2}+2d_3+2\delta\right)\cdot\frac{M^3}\ell\,.
\]
Comparing both estimates for $e_{\ev}(V_k,Q^{ij})$ we deduce
\[
\frac 12+\frac{9\eps}{10}\le \left(1+\frac\eps 5\right)\frac{\deg(P^{ij})}{\ell^2}+2d_3+2\delta\,.
\]
We may assume $d_3, \delta\le\tfrac \eps{40}$, thus getting
\[
\frac12+\frac{4\eps}{5}\le \left(1+\frac\eps 5\right)\frac{\deg(P^{ij})}{\ell^2}\,
\]
and by $\eps\ll 1$ this implies 
\[
\frac{1+\eps}2\cdot\ell^2\le \deg(P^{ij})\,,
\]  
as desired. Consequently, $\cA$ satisfies the assumptions of Lemma~\ref{K4modif} (with $\eps/2$ in place of~$\eps)$
and Theorem~\ref{K4lp} follows by an application of the embedding lemma  (Theorem~\ref{thm:EL}).
\end{proof}

\section{Concluding remarks}
\label{sec:conc}
Continuing the discussion from the introduction we mention related concepts 
of quasi\-random hypergraphs. In fact, for $3$-uniform hypergraphs 
one can define for any antichain $\ccA\neq\{\{1,2,3\}\}$ from the power set of $\{1,2,3\}$
a notion of $\ccA$-quasirandom hypergraphs (see, e.g.,~\cites{CHPS,Tow}) and
these concepts differ for non-isomorphic antichains.
Having this 
in mind, three more concepts of quasirandom hypergraphs arise, in addition to the four notions defined in 
Section~\ref{sec:intro}.  
In view of our earlier notation, we depict these three new concepts by $\VDEG$, $\VV$,
and $\EPAIR$.

\begin{dfn}
	A $3$-uniform hypergraph $H=(V, E)$ on $n=|V|$ vertices is
	\begin{enumerate}[label=\rmlabel]
	\item \emph{$(d, \eta,\VDEG)$-quasirandom}
		if for any subset of vertices $X\subseteq V$ the number 
		$e_{\vdeg}(X)$ of triples $(x,v,v')\in X\times V\times V$ with  $xvv'\in E$ satisfies
		\[
			\big|e_{\vdeg}(X)-d\,|X|\,n^2\big|\leq \eta\,n^3\,.
		\]
	\item \emph{$(d, \eta,\VV)$-quasirandom}
		if for any subsets of vertices $Y$, $Z\subseteq V$ the number 
		$e_{\vv}(Y,Z)$ of triples $(v,y,z)\in V\times Y\times Z$ with  $vyz\in E$ satisfies
		\[
			\big|e_{\vv}(Y,Z)-d\,|Y|\,|Z|\,n\big|\leq \eta\,n^3\,.
		\]
	\item \emph{$(d, \eta,\EPAIR)$-quasirandom}
		if for any subset $P\subseteq V\times V$ the number 
		$e_{\epair}(P)$ of triples $(v,y,z)\in V\times P$ with  $vyz\in E$ satisfies
		\[
			\big|e_{\epair}(P)-d\,|P|\,n\big|\leq \eta\,n^3\,.
		\]
	\end{enumerate}
\end{dfn}
With these definitions at hand we may extend the notions $\cQ(d,\eta,\star)$ and $\pi_{\star}$ to any symbol
$\star\in\{\VDEG, \VV, \EPAIR\}$ (see Definition~\ref{def:pi} and the paragraph before). It	 follows directly from the definitions that for any $d\in[0,1]$ and $\eta>0$ we have 
\begin{equation}\label{eq:Qs}
	\cQ(d,\eta,\EV)\subseteq 
	\cQ(d,\eta,\EPAIR)\subseteq
	\cQ(d,\eta,\VV)\subseteq
	\cQ(d,\eta,\VDEG)
	\qand
	\cQ(d,\eta,\VVV)\subseteq\cQ(d,\eta,\VV)\,.
\end{equation}
However, there exist examples of hypergraphs that show 
that  $\cQ(d,\eta,\VVV)$ and $\cQ(d,\eta,\EPAIR)$ are incomparable in general.
Consequently, we can extend \eqref{qr-order} for every hypergraph $F$ to
\begin{align*}
	\pi(F)\ge \pivdeg(F)\ge\pivv(F)&\ge\piepair(F)\ge\pilp(F)\ge\pill(F)\ge\pilll(F)=0\\
\intertext{and}
	\pivv(F)&\ge\pivvv(F)\ge\pilp(F)\,.
\end{align*}

Note that in the hierachy given in~\eqref{eq:Qs} the weakest concept is $\VDEG$-quasirandomness.
It follows from its definition, that any $3$-uniform $(d,\eta,\VDEG)$-quasirandom hypergraph $H=(V,E)$ 
on~$n$ vertices has the property, that all but at most $2\sqrt{\eta} n$ vertices have its 
degree in the interval $d\,n^2/2\pm\sqrt\eta\, n^2$. In fact, $\VDEG$-quasirandom hypergraphs 
are the class of hypergraphs with  approximately regular degree sequence. Owing to this, it is not hard to 
show that 
\[
	\pivdeg(F)=\pi(F)
\]
for any hypergraph $F$.

We conclude our discussion with an overview of the Tur\'an densities of the tetrahedron. For that 
the well known Tur\'an conjecture asserts 
\[
	\pi(K^{(4)}_3)=\tfrac{5}{9}
\] 
and Theorem~\ref{K4lp} and a result from~\cite{RRS-d} yield
\[
	\piev(K^{(4)}_3)=\tfrac{1}{2}\qqand\piee(K^{(4)}_3)=0\,.
\]
Moreover, Example~\ref{VR-graph} implies 
\[
	\pivvv(K^{(4)}_3)\geq \tfrac{1}{2}
	\qqand
	\piepair(K^{(4)}_3)\geq \tfrac{1}{2}
\] 
and it is tempting to conjecture for both cases that a matching upper bound holds.
Maybe an interesting first step in that direction is to combine 
both incomparable assumptions given by $\VVV$- and $\EPAIR$-quasirandomness.
\begin{quest}
	Is it true that for every $\eps>0$ there exist $\eta>0$ such that every 
	sufficiently large hypergraph 
	\[
		H\in\cQ(1/2+\eps,\eta,\VVV)\cap\cQ(1/2+\eps,\eta,\EPAIR)
	\]
	contains a tetrahedron?
\end{quest}
In view of~\eqref{hierarchy} and~\eqref{eq:Qs} a 
positive resolution of this question would strengthen our main result Theorem~\ref{K4lp}.

Finally, concerning $\VV$-quasirandomness Example~\ref{VR-graph} shows that $\pivv(K^{(3)}_4)\geq 1/2$
and the validity of Tur\'an's  conjecture would imply $\pivv(K^{(3)}_4)\leq 5/9$.

\subsection*{Acknowledgement} We thank the referees for their helpful remarks.

\end{document}